\documentclass[12pt]{amsart}
\usepackage{amsmath}
\usepackage{amscd}
\usepackage{pb-diagram}
\usepackage{comment}
\usepackage{graphicx}
\usepackage{amssymb}

\def\corcommstyle{\bf\small\tt}

\def\corrl #1<<#2||#3>>{
\if\visiblecomments y
  \begin{quote} {\corcommstyle $<<$COMMENT$>>$ #1\marginpar{!!}\\$<<$OLD$<<$} \end{quote}
  #2
  \begin{quote} {\corcommstyle $==$NEW$==$ } \end{quote}
  #3
  \begin{quote} {\corcommstyle $>>$END$>>$ } \end{quote}
 \else
  #3
 \fi
}

\long\def\longcorrl #1<<#2||#3>>{
\if\visiblecomments y
  \begin{quote} {\corcommstyle $<<$COMMENT$>>$ #1\marginpar{!!}\\$<<$OLD$<<$} \end{quote}
  #2
  \begin{quote} {\corcommstyle ==NEW== } \end{quote}
  #3
  \begin{quote} {\corcommstyle $>>$END$>>$ } \end{quote}
 \else
  #3
 \fi
}

\def\corrq #1<<#2>>{
\if\visiblecomments y
  \begin{quote} {\corcommstyle $<<$COMMENT$>>$ #1\marginpar{!!}\\$<<$BEG$<<$} \end{quote}
  #2
  \begin{quote} {\corcommstyle $>>$END$>>$ } \end{quote}
 \else
  #2
 \fi
}

\long\def\longcorrq #1<<#2>>{
\if\visiblecomments y
  \begin{quote} {\corcommstyle $<<$COMMENT$>>$ #1\marginpar{!!}\\$<<$BEG$<<$} \end{quote}
  #2
  \begin{quote} {\corcommstyle $>>$END$>>$ } \end{quote}
 \else
  #2
 \fi
}

\def\corrc #1<<>>{
\if\visiblecomments y
  \begin{quote} {\corcommstyle $<<$COMMENT$>>$ #1\marginpar{!!}} \end{quote}
\fi
}

\def\corrse #1<<>>{
\if\visiblecomments y
  \begin{quote} {\corcommstyle $<<$SECOND EDITION$>>$ #1\marginpar{!!}} \end{quote}
\fi
}

\def\corre #1<<#2||#3>>{
\if\visiblecomments y
  #3\marginpar{\corcommstyle #1}
 \else
  #3
 \fi
}

\long\def\longcorre #1<<#2||#3>>{
\if\visiblecomments y
  #3\marginpar{\corcommstyle #1}
 \else
  #3
 \fi
}

\def\corrs #1<<#2||#3>>{
\if\visiblecomments y
  #3\marginpar{\corcommstyle #2 $\rightarrow$ #3\\ #1}
 \else
  #3
 \fi
}

\def\corro #1<<#2||#3>>{
#2}

\def\corrn #1<<#2||#3>>{
#3}

\long\def\longcorro #1<<#2||#3>>{
#2}

\long\def\longcorrn #1<<#2||#3>>{
#3}

\long\def\underconstruction #1<<<#2>>>{
\if\visiblecomments y
  \begin{quote} {\corcommstyle $<<$UNDER CONSTRUCTION - BEGIN$>>$ #1\marginpar{!!}} \end{quote}
  #2
  \begin{quote} {\corcommstyle $>>$UNDER CONSTRUCTION - END$>>$ } \end{quote}
 \else
 \fi
}

\def\showcomments{
  \let\visiblecomments y
}

\def\hidecomments{
  \let\visiblecomments n
}

\def\refeq#1{\if\workingver y(\ref{#1})-[[#1]]\else(\ref{#1})\fi}
\def\refth#1{\if\workingver y\ref{#1}-[[#1]]\else\ref{#1}\fi}
\def\mylabel#1{\if\workingver y\label{#1}{\bf\ \ [[#1]]\ \ }\else\label{#1}\fi}
\def\mybibitem#1{\if\workingver y\bibitem{#1}{\bf\ \ [[#1]]\ \
}\else\bibitem{#1}\fi}

\newfont{\msam}{msam10}
\newfont{\msbm}{msbm10}

\def\articletheorems{
\newtheorem{thm}{Theorem}[section]
\newtheorem{lem}[thm]{Lemma}

\newtheorem{cor}[thm]{Corollary}

\newtheorem{ex}[thm]{Example}
\newtheorem{algo}{Algorithm}[section] 
\newtheorem{alg}[thm]{Algorithm}  

}

\def\cK{\text{$\mathcal K$}}

\def\cS{\text{$\mathcal S$}}
\def\cT{\text{$\mathcal T$}}

\newcommand{\cl}{\operatorname{cl}}
\newcommand{\opn}{\operatorname{opn}}

\newcommand{\card}{\operatorname{card}}

\newcommand{\im}{\operatorname{im}}

\def\begeq#1{\begin{equation}\mylabel{#1}}
\def\endeq{\end{equation}}

\def\mathobj#1{\mbox{$#1$}}

\def\ZZ{\mathobj{\mathbb{Z}}}

\def\abs#1{|#1|}

\def\setof#1{\mbox{$\{\,#1\,\}$}}

\def\0#1{\hbox{\kern25pt}$ #1 $\\}
\def\1#1{\hbox{\kern40pt}$ #1 $\\}
\def\2#1{\hbox{\kern55pt}$ #1 $\\}
\def\3#1{\hbox{\kern70pt}$ #1 $\\}

\newcounter{li}

\def\begalg#1{\begin{algo}\mylabel{#1}\normalshape:\small\baselineskip 10pt\\}
\def\endalg{\end{algo}}

\def\Figures(include=#1,cat=#2){
  \renewcommand{\textfraction}{.20}
  \renewcommand{\topfraction}{.80}
  \renewcommand{\bottomfraction}{.80}
  \renewcommand{\floatpagefraction}{.80}
  \newcount\figcount
  \figcount=0
  \let\includefigures=#1
  \def\figcat{#2}
}

\def\FigureFromFile[#1][#2](#3)#4
{
  \begin{figure}[htbp]
     \global\advance\figcount by 1
     \if\includefigures y\special{anisoscale #1.wmf, \the\hsize #2}\fi
     \vspace{#2}
     \caption{#4}
     \mylabel{#3}
   \end{figure}
}

\def\FigureFromFileTwoD[#1][#2,#3](#4)#5
{
  \begin{figure}[htbp]
     \global\advance\figcount by 1
     \if\includefigures y\special{anisoscale #1.wmf, #2 #3}\fi
     \vspace{#2}
     \caption{#5}
     \mylabel{#4}
   \end{figure}
}

\def\FigureF<#1>[#2](#3)#4
{
  \begin{figure}[htbp]
     \global\advance\figcount by 1
     \if\includefigures y\special{anisoscale \figcat/fig\number\figcount.wmf,
       \the\hsize #2}
     \fi
     \if\includefigures p
       \leavevmode
       \epsfxsize=\hsize
       \epsffile{#1}
     \fi
     \if\includefigures y
          \vspace{#2}
     \fi
     \caption{#4}
     \mylabel{#3}
   \end{figure}
}

\def\Figure[#1](#2)#3
{
  \begin{figure}[htbp]
     \global\advance\figcount by 1
     \if\includefigures y\special{anisoscale \figcat/fig\number\figcount.wmf,
       \the\hsize #1}
     \fi
     \if\includefigures p
       \leavevmode
       \epsfxsize=\hsize
       \epsffile{fig\number\figcount.eps}
     \fi
     \if\includefigures y
          \vspace{#1}
     \fi
     \caption{#3}
     \mylabel{#2}
   \end{figure}
}

\let\visiblecomments y

\graphicspath{
  {../../../images/auxil/}
  {../../../images/bargraphs/}
  {../../../images/capd/}
  {../../../images/cechstruct/}
  {../../../images/conleyIndex/}
  {../../../images/cwcomplex/}
  {../../../images/discMorse/}
  {../../../images/electromagn/}
  {../../../images/homology/}
  {../../../images/horseshoe/}
  {../../../images/imaging/}
  {../../../images/lorenz/}
  {../../../images/matrices/}
  {../../../images/mvmaps/}
  {../../../images/people/}
  {../../../images/persistence/}
  {../../../images/reductions/}
  {../../../images/repsets/}
  {../../../images/segments/}
  {../../../images/sensors/}
  {../../../images/weather/}
}

\def\adhl{\prec}

\usepackage{commath}
\usepackage{amsfonts}
\usepackage[T1]{fontenc}
\usepackage[english]{babel}
\usepackage[utf8]{inputenc}

\usepackage{tikz}
\usetikzlibrary{matrix}
\usepackage{amssymb}
\def\PreOrd{\textsc{PreOrd}}
\def\POrd{\textsc{PartOrd}}
\def\Top{\textsc{Top}}
\newcommand{\mo}{\operatorname{mo}}

\newtheorem*{theorem*}{Theorem}

\articletheorems

\begin{document}

\author{Jacek Kubica}
\address{Jacek Kubica,
  Division of Computational Mathematics,
  Faculty of Mathematics and Computer Science,
  Jagiellonian University, ul.~St. \L{}ojasiewicza 6, 30-348~Krak\'ow, Poland
}
\email{jacek.kubica@student.uj.edu.pl}
\author{Marian Mrozek}
\address{Marian Mrozek,
  Division of Computational Mathematics,
  Faculty of Mathematics and Computer Science,
  Jagiellonian University, ul.~St. \L{}ojasiewicza 6, 30-348~Krak\'ow, Poland
}
\email{Marian.Mrozek@ii.uj.edu.pl}
\thanks{This research is partially supported
       by the Polish National Science Center under Ma\-estro Grant No. 2014/14/A/ST1/00453.}

\title[Lefschetz Complexes as Finite Topological Spaces]{Lefschetz Complexes as Finite Topological Spaces}

\begin{abstract}
   We consider a fixed basis of a finitely generated free chain complex
   as a finite topological space and we present a sufficient condition
   for the singular homology of this space to be isomorphic with the homology of the chain complex.
\end{abstract}

\maketitle

\section{Introduction}

The combinatorial Morse theory, originally developed for finite CW complexes by R. Forman \cite{Fo98a} in 1998,
was generalized to the purely algebraic setting of finitely generated free chain complexes with a distinguished basis
by Kozlov \cite{Ko} in 2005, and independently by Skj\"oldberg \cite{Sk} in 2006 and  J\"ollenbeck and Welker \cite{JW} in 2009.
A finitely generated free chain complex with a distinguished basis was already studied by Lefschetz in 1942.
However, in the original Lefschetz definition \cite[Chpt. III, Sec. 1, Def. 1.1]{Le1942}
the elements of the basis are the primary objects of interest
and the algebraic structure of the associated chain complex is given on top of them (see Section~\ref{sec:lefschetz} for a precise definition).

Thus, the Lefschetz complex consists of a finite collection of cells $X$ graded by dimension
and the {\em incidence coefficient} $\kappa(x,y)$ encoding the incidence relation
between cells $x,y\in X$.
This way of viewing a free chain complex with a distinguished basis
is convenient in the algorithmic context, because a computer may store and process only finite sets. In particular, in recent years, Lefschetz complexes became an object of interest as a convenient tool in computational dynamics and topology \cite{HMMN, HMS, MM}.
As examples of Lefschetz complexes let us mention: the family $K$ of all simplices of a simplicial complex \cite[Definition 11.8]{KMM2004},
all elementary cubes in a cubical set \cite[Definition 2.9]{KMM2004}
or, more generally, cells of a cellular complex (finite CW complex, see \cite[Section IX.3]{Ma1991}).
A non-zero value of $\kappa(x,y)$
indicates that the cell $y$ is in the boundary of the cell $x$ and the dimension
of $y$ is one less than the dimension of $x$. The cell $y$ is then called a {\em facet} of $x$.
The linear combination of facets of a cell $x$ with incidence coefficients as coefficients is the boundary of $x$.
Under suitable assumptions this defines the associated free chain complex with basis $X$ and the homology of $X$
which we will refer to as the Lefschetz homology of $X$.

The relation $y\adhl_{\kappa} x$ defined by $\kappa(x,y)\neq 0$ extends to a partial order $\leq_\kappa$ on $X$.
Hence, by Alexandroff Theorem \cite{Al1937}, a Lefschetz complex is also a finite topological space. Therefore, also its singular homology is well defined. This leads to a natural question under what conditions the Lefschetz homology is isomorphic to the singular homology. In this paper we give a partial answer to this question by proving the following theorem (see Theorem \ref{thm:main_thm} for the precise formulation).

\begin{theorem*}
Let $(X, \kappa)$ be an augmentable Lefschetz complex such that the Lefschetz homology of the closure of every point in $X$ coincides with the singular homology of the one-point space. Then the Lefschetz homology and the singular homology of $X$ are isomorphic.
\end{theorem*}

The organization of the paper is as follows.
In Section 2  we recall some preliminary notions and theorems about chain complexes, finite topological spaces and Lefschetz complexes.
In Section 3 we state and prove the main result of this paper.

\section{Preliminaries}
In this section we recall definitions and results needed in the sequel and we fix the notation.
\subsection{Chain complexes.}

Recall that given a graded free module $C =  (C_n)_{n \in \mathbb{Z}}$ (i.e. $C_n$ is a free module for each $n \in \mathbb{Z}$) and a sequence $\partial = (\partial_n : C_n \rightarrow C_{n-1})_{n \in \mathbb{Z}}$ such that for each $n \in \mathbb{Z}$ we have $\partial_n \partial_{n-1} = 0$, the pair $(C, \partial)$ is called a \emph{chain complex} and $\partial$ is called the associated \emph{boundary map}.
\par
Given a chain complex $(C, \partial)$  and an $n \in \mathbb{Z}$ we write $B_n(C) := \im \partial_{n + 1}$ for the $n$\emph{-boundaries} and $Z_n(C) := \ker \partial_n$ for the $n$\emph{-cycles}. The fact that $B_n(C)$ is a subgroup of $Z_n(C)$ lets us define the $n$\emph{-th homology group} as $H_n(C) := Z_n(C) / B_n(C)$.
\par
Let $(C, \partial)$ and  $(C', \partial')$ be chain complexes. A sequence $(\varphi_n: C_n \rightarrow C'_n)_{n \in \mathbb{Z}}$ of maps is called a \emph{chain map} if $\varphi_{n-1} \partial_n = \partial_n \varphi_n$ for each $n \in \ZZ$.
\par 
In the sequel we drop the subscripts in boundary maps and chain maps whenever they are clear from the context.
\par 
The following excision theorem for simplicial complexes will be used in the sequel.

\begin{thm} \label{thm:excision_simplicial}
\cite[Ch. V, Sec. 2, Thm. 4]{Duda}
If $\cK, \cK_1, \cK_2$ are simplicial complexes such that $\cK = \cK_1 \cup \cK_2$, then $H(|\cK_2|, |\cK_1 \cap \cK_2|) = H(|\cK|, |\cK_1|)$.
\end{thm}

\subsection{Finite topological spaces.}
Recall that a \emph{finite topological space} is a pair  $(X, \cT)$ such that $X$ is a finite set and $\cT$ is a topology on $X$.
\par 
Note that, unlike for the general topological spaces, if $(X, \cT)$ is a finite topological space and $A \subset X$, then $$\opn _{\cT} A := \bigcap \{ U \in \cT \mid A \subset U \}$$ is open.
\par
To keep the notation simple we write $\cl_\cT x := \cl_\cT \{ x \}, \opn_\cT x := \opn_\cT \{ x \}$ for any $x \in X$. Also, whenever topology $\cT$ is clear from the context, we drop the subscript $\cT$ in $\cl_\cT$ and $\opn_\cT$.

Consider a finite set $X$. Denote by $\PreOrd(X)$ the family of all preorders on $X$,
by $\POrd(X)$ the family of all partial orders on $X$, by $\Top(X)$ the family of all topologies on $X$ and by $\Top_0(X)$ the subfamily of $T_0$ topologies on $X$. For a preorder $\le$ in $\PreOrd(X)$ and a subset $A \subset X$ the set $A^\le := \{ x \in X \mid \exists a\in A\;x\leq a \}$ is called the \emph{lower set of $A$} and $A^\ge := \{ x \in X \mid \exists a\in A\;a\leq x \}$ is called the \emph{upper set of $A$}. For convenience we set $x^\le := \{ x \} ^\le$ and $x^\ge := \{ x \} ^\ge$. We say that $A$ is an \emph{upper (respectively lower) set} if $A = A^\le$ (respectively $A = A^\ge$).
\par
For a preorder $\le$ in $\PreOrd(X)$ set
\[
  \cT_\le:=\setof{A\subset X\mid \text{ $A$ is an upper set with respect to $\le$}}
\]
and for a topology $\cT\in\Top(X)$ set
\[
   \le_\cT:=\setof{(x,y)\in X^2\mid x\in\cl_{\cT}y}.
\]

Finite topological spaces can be characterized in terms of preorders as in the following theorem which goes back to P.S.~Alexandroff~\cite{Al1937}.
\begin{thm}
\label{thm:alexandroff}
  Let $X$ be a finite set. We have the following properties.
\begin{itemize}
   \item[(i)] For every preorder $\le$ in $\PreOrd(X)$ the family $\cT_\le$ is a topology on $X$. Moreover,
   for every $x\in X$
\begin{equation}
  x^\ge=\opn_{\cT_\le}x.
\end{equation}
   \item[(ii)] For every topology $\cT\in\Top(X)$ the relation $\ge_{\cT}$ defined by $x\in\cl y$ is a preorder on $X$.
   Moreover, for any $x\in X$
\begin{equation}
 x^{\ge_\cT}=\opn_{\cT}x.
\end{equation}
   \item[(iii)] The  maps
\begin{eqnarray*}
  \alpha: \PreOrd(X)\ni \le&\mapsto&\cT_{\le}\in\Top(X),\\
  \beta: \Top(X)\ni\cT&\mapsto&\le_{\cT}\in \PreOrd(X)
\end{eqnarray*}
   are mutually inverse bijections under which the partial orders on $X$ correspond to $T_0$ topologies on $X$.
\end{itemize}
\end{thm}
\begin{thm}
\cite[Prop. 1.2.1]{Ba2011}
Assume $(X,\cT), (Y,\cS)$ are finite topological spaces. The map $f:(X,\cT) \rightarrow (Y,\cS)$ is continuous if and only if the map $f:\left(X,\le_\cT\right) \rightarrow (Y,\le_\cS)$ is order preserving.
\end{thm}

\begin{thm}
\cite[Rem. 1.2.8]{Ba2011}
\label{thm:contractible_closure}
Let $(X, \cT)$ be a finite $T_0$ topological space. Then $\cl x$ is contractible for any $x \in X$.
\end{thm}

Now assume $(X, \cT)$ is a finite $T_0$ space and denote by $\cK(X)$ the abstract simplicial complex of subsets of $X$ linearly ordered by $\le_\cT$ and its geometric realization by $\abs{\cK(X)}$. The \emph{McCord map} is the map $\mu:\abs{\mathcal{K}(X)} \rightarrow X$ defined for a point $p = t_1 x_1 + ... + t_k x_k \in \abs{\cK(X)}$ by $\mu(p) := x_1$ where $x_1 \le_{\cT} x_2 \le_{\cT} ... \le_{\cT} x_k$.
\begin{thm}
\emph{(McCord, }\cite{McCord}\emph{)}
For every finite $T_0$ topological space $X$ the McCord map $\mu$ is a weak homotopy equivalence.
\end{thm}
\begin{cor}
\label{cor:mccord}
For every finite $T_0$ space $X$ the McCord map $\mu$ induces an isomorphism $\mu_* : H(|\cK(X)|) \rightarrow H(X)$.
\end{cor}

Using elementary results from homological algebra we obtain the following relative version of McCord Theorem.

\begin{cor}
\label{cor:rel_mccord}
Let $X$ be a finite $T_0$ space and let $A \subset X$ be its subspace. Then the McCord map $\mu$ induces an isomorphism from $H(X, A)$ to $H(|\mathcal{K}(X)|, |\mathcal{K}(A)|)$.
\end{cor}

\subsection{Lefschetz complexes.}
\label{sec:lefschetz}

Let $R$ denote a ring with unity. We say that $(X,\kappa)$ is a {\em Lefschetz complex} (see \cite[Chpt. III, Sec. 1, Def. 1.1]{Le1942} and compare also with definition of an \emph{S-complex} in \cite[Section 2]{MB}).
If $X=(X_q)_{q\in{\scriptsize \mathbb{Z}^+}}$ is a finite set with gradation,
$\kappa : X \times X \to R$ is a map such that $\kappa(x,y)\neq 0 $
implies $x\in X_q,\,y\in X_{q-1}$ and for any $x,z\in X$ we have
\begin{equation}
\label{eq:kappa-condition}
    \sum_{y\in X}\kappa(x,y)\kappa(y,z)=0.
\end{equation}
We refer to the elements of $X$ as {\em cells}
and to $\kappa(x,y)$ as the {\em incidence coefficient} of $x,y$. If $x \in X_q$, we say that $x$ has \emph{dimension} $q$.
\par
Let $(X,\kappa)$ be a fixed Lefschetz complex. Denote by $C^\kappa(X)$ the graded free module over $R$ spanned by $X$. We refer to the elements of $C^\kappa(X)$ as \emph{Lefschetz chains}. Define $\partial^\kappa:C^\kappa(X) \rightarrow C^\kappa(X)$ on generators $x \in X$ by $\partial^\kappa(x) := \sum_{y \in X} \kappa(x, y) y$. By (\ref{eq:kappa-condition}) we easily see that $(C^\kappa(X), \partial^\kappa)$ is a chain complex. We denote its $n$-th homology group by $H_n^\kappa(X)$ and call it $n$\emph{-th Lefschetz homology group} of $X$.
\par
Note that the concept of Lefschetz complex is complementary to that of chain complex: it shifts focus to the basis of the complex. Indeed, if $(C, \partial)$ is a finitely generated cell complex and $(e_i)_{i=1}^n$ is its fixed basis, then by taking $X = \{ e_i \mid i = 1,\ldots,n \}$ and setting $\kappa(e_i, e_j)$ as the coefficient of $e_j$ in the decomposition of $\partial e_i$ on the basis, we obtain Lefschetz complex $X$ such that $(C^\kappa(X), \partial^\kappa) = (C, \partial)$.
\par
If for $x, y \in X$ we have $\kappa(x, y) \neq 0$, we say that $y$ is a \emph{facet} of $x$ and we write $y \prec x$. We denote by $\leq$ the transitive and reflexive closure of $\prec$, that is the smallest with respect to inclusion relation that contains $\prec$ and is both transitive and reflexive. If $y \leq x$ we say that $y$ is a \emph{face} of $x$.
\par
By Theorem \ref{thm:alexandroff} we have the $T_0$-topology $\cT_\leq$ associated with $\left(X, \leq\right)$. We call it the \emph{Lefschetz topology} of $X$. By $C(X)$ we denote the chain complex of singular chains on $(X, \cT_\leq)$. The chain complexes of singular chains $C(X)$ should not be mistaken with the chain complex of Lefschetz chains, denoted by $C^\kappa(X)$.
\par
Now, we recall a few definitions and theorems which will be used in the sequel.
\par
Given $A \subset X$ we define the \emph{mouth} of $A$ as the set $\mo A := \cl A \setminus A$. We say that $A$ is \emph{locally closed} if $\mo A$ is closed. 
Note that in \cite{MM} a locally closed set is called proper. 
Observe that, in particular, closed and open sets are locally closed.

\begin{thm} \label{cor:lef_subcomplex}
\cite[Theorem 3.1]{MB}
Let $(X, \kappa)$ be a Lefschetz complex and let $A \subset X$ be locally closed in the Lefschetz topology of $X$. Then $(A, \kappa \restriction_{A \times A})$ is a Lefschetz complex.
\end{thm}

\begin{thm} \label{thm:long_exact_lef}
\cite[Theorem 3.4]{MB}
Assume $X$ is a Lefschetz complex and $X' \subset X$ is closed. Let $i: C^\kappa(X') \rightarrow C^\kappa(X)$ denote inclusion and $j: C^\kappa(X) \rightarrow C^\kappa(X \setminus X')$ be the projection defined by 
$$
j(c) =
\begin{cases}
0, & c \in C^\kappa(X'), \\
c, & \textnormal{otherwise}.
\end{cases}
$$
Then the following sequence

\[
\begin{diagram}
\node{0} \arrow{e}{} \node{C^\kappa(X')} \arrow{e,t}{i} \node{C^\kappa(X)} \arrow{e,t}{j} \node{C^\kappa(X \setminus X')} \arrow{e}{} \node{0}
\end{diagram}
\]
is exact and it extends to a long exact sequence.
\end{thm}
\begin{thm}
\label{thm:lefschetz_excision}
\cite[Theorem 3.5]{MB}
Let $(X, \kappa)$ be a Lefschetz complex and $X' \subset X$ be closed in the Lefschetz topology of $X$. Then the homology of $H^\kappa(X, X')$ the quotient chain complex $(C^\kappa(X) / C^\kappa(X'), \partial^\kappa)$ is isomorphic to $H^\kappa(X \setminus X')$.
\end{thm}

\section{Main result}
Let $(X, \kappa)$ be a Lefschetz complex over a ring $R$. We say that $(X, \kappa)$ is \emph{augmentable}, if the linear map $\epsilon^\kappa:C_0^\kappa \rightarrow R$ defined on the basis elements $x \in X_0$ by $\epsilon^\kappa(x) = 1$ satisfies $\epsilon^\kappa \partial_1^\kappa = 0$. The following lemma may be deduced from the Acyclic Carrier Theorem \cite[Thm. 13.4]{Munkres}. 
We include a direct proof for convenience.
\begin{lem}
\label{lem:main_lemma}
Assume that $(X, \kappa)$ is an augmentable Lefschetz complex. Then there exists a chain map $\varphi:C^\kappa(X) \rightarrow C(X)$ from Lefschetz chains to singular chains such that $\varphi_0$ on the basis element $x \in X_0$ is the singular chain 
\begin{equation} \label{eq:lemma_help_eq}
c_x : \Delta _0 = \{ 1 \} \ni 1 \mapsto x \in X_0
\end{equation}
and
\begin{equation} \label{eq:lemma_eq}
\forall{c \in C^\kappa(X)} \ |\varphi(c)| \subset \cl |c|.
\end{equation}
\end{lem}
\begin{proof}
For $k < 0$ we set $\varphi_k := 0$ and we define $\varphi_0$ on the basis element $x \in X_0$ as $c_x \in C_0(X)$ where $c_x$ is given by (\ref{eq:lemma_help_eq}). Obviously $\partial_0 \varphi_0 = 0 = \varphi_{-1} \partial_0^\kappa$. Also (\ref{eq:lemma_eq}) is trivially satisfied for $k = 0$.
\par 
Thus, consider $k > 0$ and assume that $\varphi:C_{k-1}^\kappa(X) \rightarrow C_{k-1}(X)$ satisfying $\partial_{k-1} \varphi_{k-1} = \varphi_{k-2} \partial_{k-1}^\kappa$ and (\ref{eq:lemma_eq}) is already defined. For $x \in X_k$ set $u_x := \varphi_{k-1}\partial_k^\kappa x \in C_{k-1}(X)$. Then $\partial_{k-1}u_x = \partial_{k-1}\varphi_{k-1}\partial_k^\kappa x = \varphi_{k-2} \partial_{k-1}^\kappa\partial_k^\kappa x = 0$. Hence, $u_x$ is a singular $(k-1)$-cycle. Since $\varphi_{k-1}$ satisfies (\ref{eq:lemma_eq}), we also have $|u_x| = |\varphi_{k-1}\partial_k^\kappa x| \subset \cl |\partial_k^\kappa x| \subset \cl \cl x = \cl x$. It follows that $u_x \in Z_{k-1}(\cl x)$. If $k > 1$, then $u_x \in B_{k-1}(\cl x)$, because by Theorem \ref{thm:contractible_closure} the set $\cl x$ is contractible, which means that $H_l(\cl x) = 0$ for $l > 0$. Hence, consider the case $k = 1$. Let $\epsilon :C_0(X) \rightarrow R$ denote the augmentation for the singular chains. Then, for $x \in X_0$ we have $\epsilon \varphi_0 x = \epsilon c_x = 1 = \epsilon^\kappa x$. It follows, that we have $\epsilon \varphi_0 = \epsilon^\kappa$ and we get $\epsilon u_x = \epsilon \varphi_0 \partial_1^\kappa x = \epsilon^\kappa \partial_1^\kappa x = 0$ by the assumed augmentability of $(X, \kappa)$. Moreover, since $\cl x$ is contractible, we have $\ker \epsilon = \im \partial_1$. It means that $u_x \in B_0(\cl x)$.
\par
Thus, we have proved that for every $k > 0$ and $x \in X_k$ there exists a chain $s_x \in C_k(\cl x)$ such that $u_x = \partial s_x$. Set $\varphi_k(x) := s_x$ for $x \in X_k$ and extend it linearly to $\varphi_k:C_k^\kappa(X) \rightarrow C_k(X)$. We have $|\varphi _k(x)| = |s_x| \subset \cl x$ and $\partial_k \varphi_k x = \partial_k s_x = u_x = \varphi_{k-1} \partial_k x$ for every $x \in X_k$. Now consider a Lefschetz chain $c = \sum_{i = 1}^n \alpha_i x_i \in C_k^\kappa(X)$ where $n \in \mathbb{N}, \alpha_i \in R, x_i \in X_k$ for $i = 1,...,n$. We have $|\varphi_k(c)| \subset \bigcup_{i = 1}^n |\varphi_k (x_i)| \subset \bigcup_{i = 1}^n \cl x_i = \cl \bigcup_{i = 1}^n \{x_i\} = \cl |c|$. Therefore, $\varphi_k$ satisfies (\ref{eq:lemma_eq}). We also have $\partial_k \varphi_k (c) = \sum_{i=1}^n \alpha_i \partial_k \varphi_k(x_i) = \sum_{i=1}^n \alpha_i \varphi_{k - 1} \partial_k (x_i) = \varphi_{k-1} \partial_k \sum_{i=1}^n \alpha_i x_i = \varphi_{k-1} \partial_k c$. This concludes the proof.
\qed
\end{proof}

Now we are ready to state the main theorem of this paper.

\begin{thm} \label{thm:main_thm}
Let $(X, \kappa)$ be an augmentable Lefschetz complex such that for every $x \in X$ we have
\begin{equation}
\label{eq:main-thm-eqn}
H^\kappa_n(\cl x) = 
\begin{cases}
R & n = 0, \\
0 & \textnormal{otherwise.}
\end{cases}
\end{equation}
Then $H^\kappa(X) \cong H(X)$.
\end{thm}
\begin{proof}
Let $\varphi: C^\kappa(X) \rightarrow C(X)$ be a chain map such as in Lemma \ref{lem:main_lemma}. We will prove by induction with respect to $k := \card X$ that $\varphi_*$ is an isomorphism. If $k = 1$, then the only element in $X$ is of dimension zero, because otherwise (\ref{eq:main-thm-eqn}) is not satisfied and the conclusion follows immediately. Assume now that $k > 1$ and that thesis holds for the Lefschetz complexes of cardinality $k' < k$. Let $a$ be a maximal element in $X$ with respect to $\leq$ and set $X' := X \setminus \{a\}$. Then $X'$ is closed in $X$. Thus, by  Theorem \ref{cor:lef_subcomplex} we have a well defined Lefschetz complex $(X', \kappa_{X' \times X'})$. It is augmentable, because its boundary operator is the restriction of the boundary operator on $X$, and it satisfies (\ref{eq:main-thm-eqn}), because maximality of $a$ implies that for every $x \in X'$ its closure is the same in $X$ and $X'$. Therefore, $(X', \kappa_{X' \times X'})$ satisfies the inductive assumption.
\par 
Observe that since \ref{eq:lemma_eq} holds, $\varphi(C^\kappa(X')) \subset C(X')$ and consequently $\varphi$ induces a chain map $\varphi_*: H^\kappa(X, X') \rightarrow H(X, X')$. Thus, by Lemma \ref{lem:main_lemma} and Theorem \ref{thm:long_exact_lef} 
we have the following commutative diagram with exact rows

\[
\begin{diagram}
\label{dg:main-thm-diagram}
	\node{\ldots} \arrow{e,t}{\partial_*} \node{H_n(X')} \arrow{e,t}{i_*}  \node{H_n(X)} \arrow{e,t}{j_*} \node{H_n(X, X')} \arrow{e,t}{\partial_*} \node{\ldots} \\
	\node{\ldots} \arrow{e,t}{\partial_*} \node{H^\kappa_n(X')} \arrow{e,t}{i_*}\arrow{n,r}{\varphi_*}  \node{H^\kappa_n(X)} \arrow{e,t}{j_*}\arrow{n,r}{\varphi_*} \node{H^\kappa_n(X, X')} \arrow{e,t}{\partial_*}\arrow{n,r}{\varphi_*} \node{\ldots} 
\end{diagram}.
\]

\par
By the induction assumption $\varphi_*: H_n^\kappa(X') \rightarrow H_n(X')$ is an isomorphism for every $n \in \mathbb{Z}$. 
\par
By Corollary \ref{cor:rel_mccord} we have $H_n(X, X') \cong H_n(|\mathcal{K}(X)|,| \mathcal{K}(X')|)$. Observe that $\mathcal{K}(\cl a)$ is a subcomplex of $\mathcal{K}(X)$ such that $\mathcal{K}(X) = \mathcal{K}(X') \cup \mathcal{K}(\cl a)$. Hence, $H_n(\mathcal{K}(X), \mathcal{K}(X')) \cong H_n(\mathcal{K}(\cl a), \mathcal{K}(\cl a) \cap \mathcal{K}(X')) = H_n(\mathcal{K}(\cl a), \mathcal{K}(\mo a))$ by Theorem \ref{thm:excision_simplicial}. Thus, using Corollary \ref{cor:rel_mccord} again we get $H_n(X, X') \cong H_n(\cl a, \mo a)$. We also get from Theorem \ref{thm:lefschetz_excision} that 
$$H_n^\kappa(X, X') \cong H_n^\kappa(X \setminus X') = H_n^\kappa(a) = H_n^\kappa(\cl a \setminus \mo a) \cong H_n^\kappa(\cl a, \mo a).$$
Now consider the diagram
\[
\begin{diagram}
	\node{\ldots} \arrow{e,t}{\partial_*} \node{H_n(\mo a)} \arrow{e,t}{i_*}  \node{H_n(\cl a)} \arrow{e,t}{j_*} \node{H_n(\cl a, \mo a)} \arrow{e,t}{\partial_*} \node{\ldots} \\
	\node{\ldots} \arrow{e,t}{\partial_*}  \node{H^\kappa(\mo a)} \arrow{e,t}{i_*} \arrow{n,r}{\varphi_*}  \node{H^\kappa_n(\cl a)} \arrow{e,t}{j_*}\arrow{n,r}{\varphi_*} \node{H^\kappa_n(\cl a, \mo a)} \arrow{e,t}{\partial_*} \arrow{n,r}{\varphi_*} \node{\ldots} 
\end{diagram}.
\]

The homomorphism $\varphi_*: H_{n}^\kappa(\mo a) \rightarrow H_{n}(\mo a)$ is an isomorphisms for $n \in \mathbb{Z}$ by the induction assumption. For $n \in \mathbb{Z} \setminus \{1\}$ map $\varphi_* : H_n^\kappa(\cl x) = 0 \rightarrow H_n(\cl x) = 0$ is obviously an isomorphism and for $n = 0$ map $\varphi_* : H_n^\kappa(\cl x) = R \rightarrow H_n(\cl x) = R$ is an isomorphism straightforwardly from the definition of $\varphi$. By the Five Lemma (see \cite[Lemma 7.1]{Ma1991}) it follows that $\varphi_*: H_{n}^\kappa(\cl a, \mo a) \rightarrow H_{n}(\cl a, \mo a)$ is an isomorphism. Again by the Five Lemma we obtain from diagram \ref{dg:main-thm-diagram} that $\varphi_*: H_n^\kappa(X) \rightarrow H_n(X)$ is also an isomorphisms. This completes the proof.
\qed
\end{proof}

Note that a closed subcomplex of the augmentable Lefschetz complex is itself augmentable. Thus, we easily obtain the following corollary

\begin{cor}
\label{cor:main-thm-cor}
Let $(X, \kappa)$ be an augmentable Lefschetz complex. Then $(X, \kappa)$ satisfies (\ref{eq:main-thm-eqn}) if and only if for every closed subcomplex $X' \subset X$ we have $H^\kappa(X') \cong H(X')$.
\end{cor}

Corollary \ref{cor:main-thm-cor} does not resolve the problem whether inverse to Theorem \ref{thm:main_thm} is true. This question remains open.
\par
Now we provide some examples, which show that assumptions of Theorem \ref{thm:main_thm} are essential. Here we take our ring to be $R = \ZZ$.

\begin{ex} \label{ex1}
Consider a Lefschetz complex $(X, \kappa)$ with $X_0 = \{a,b,c,d\}$, $X_1 = \{e\}$ and incidence coefficients $\kappa(e, a) = 1, \kappa(e, b) = 1, \kappa(e, c) = -1, \kappa(e, d) = -1$ given by the matrix
\begin{center}
  \begin{tabular}{ c | r }
      & e \\ \hline
    a & 1 \\ \hline
    b & 1 \\ \hline
    c & -1 \\ \hline
    d & -1 \\
  \end{tabular}.
\end{center}
This Lefschetz complex and the associated simplicial complex is visualized in Figure 1. It is elementary to observe that $X$ is augmentable and
\begin{equation*}
H^\kappa_n(\cl e) = H^\kappa_n(X) =
\begin{cases}
R^3 & n = 0, \\
0 & \textnormal{otherwise.}
\end{cases},
\end{equation*}
However, by Theorem \ref{thm:contractible_closure} we have
\begin{equation*}
H_n(X) =
\begin{cases}
R & n = 0, \\
0 & \textnormal{otherwise.}
\end{cases}
\end{equation*}

\end{ex}

\begin{ex} \label{ex2}
Consider a Lefschetz complex $(X, \kappa)$ with $X_0 = \{a,b\}$, $X_1 = \{c, d\}$ and incidence coefficients given by the matrix

\begin{center}
  \begin{tabular}{ l | c | r  }
      & c & d \\ \hline
    a & 1 & 1 \\ \hline
    b & -1 & 1 \\
  \end{tabular}
\end{center}
This Lefschetz complex and the associated simplicial complex is visualized in Figure 1. One easily verifies that $X$ satisfies assumption (\ref{eq:main-thm-eqn}), but it is not augmentable. This Lefschetz complex has no non-trivial cycles in dimension $1$. Indeed, assume that $z := \gamma c + \delta d \in Z_1^\kappa(X)$ where $\gamma, \delta \in \ZZ$. It follows that $0 = \partial z = (\gamma + \delta) a + (\delta - \gamma) b$. This is true only if $\gamma = \delta = 0$. Thus, $z = 0$ and consequently $H_1^\kappa(X) = 0$. However, by Corollary (\ref{cor:mccord}) we easily obtain that $H_1(X) \cong H_1(\mathbb{S}^1) \cong R$.
\end{ex}

\begin{figure}[h!]
\includegraphics[width=1\textwidth]{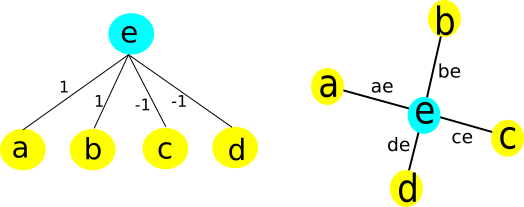}

\caption{Left: poset visualization of Example \ref{ex1}. Right: simplicial complex associated with it. The set on the right is contractible, as one would expect.}
\end{figure}

\begin{figure}[h!]
\includegraphics[width=1\textwidth]{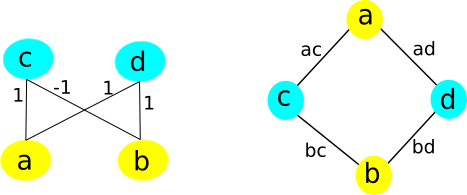}

\caption{Left: poset visualization of Example \ref{ex2}. Right: simplicial complex associated with it.}
\end{figure}

\end{document}